\documentclass[10pt]{amsart}
      \makeatletter
     \def\section{\@startsection{section}{1}%
     \z@{.7\linespacing\@plus\linespacing}{.5\linespacing}%
     {\bfseries
     \centering
     }}
     \def\@secnumfont{\bfseries}
     \makeatother
\newtheorem{theorem}{Theorem}[section]
\newtheorem{lemma}[theorem]{Lemma}
\newtheorem{proposition}[theorem]{Proposition}
\newtheorem{corollary}[theorem]{Corollary}
\theoremstyle{definition}

\newtheorem{example}[theorem]{Example}
\theoremstyle{remark}
\newtheorem{remark}[theorem]{Remark}
\numberwithin{equation}{section}
\def \a{{\alpha}}

\def \D{{\Delta}}

\def \d{{\delta}}
\def \e{{\varepsilon}}

\def \g{{\gamma}}
\def \G{{\Gamma}}

\def \l{{\lambda}}

\def \p{{\varphi}}
\def \t{{\vartheta}}

\def \m{{\mu}}
\def \s{{\sigma}}

\def \C{{\cal C}}

\def \E{{\bf E}\, }

\def \P{{\bf P}}

\def \qq{{\qquad}}
\def \R{{\bf R}}

\def \T{{\bf T}}

\def \uc{{\underline{c}}}

\def \uX{{\underline{X}}}
\def \uY{{\underline{Y}}}

\def \Z{{\bf Z}}

\def \dd{{\rm d}}
\def \noi{{\noindent}}

\def\E{{\mathbb E \,}}

\def \T{{\mathbb T}}

\def\P{{\mathbb P}}

\def\R{{\mathbb R}}

\def\Z{{\mathbb Z}}

\def\C{{\mathbb C}}  
  
\def\T{{\mathbb T}}  

 \font\phh=cmr10 at  8,2 pt \scrollmode

\font\gsec= cmb10 at 10 pt
   at 10,8  pt 
    \scrollmode
\font\sec= cmb10 at 10  pt
\font\gsec= cmb10 at 9,49  pt
\font\gssec= cmb10 at 8,3  pt

 \title[Small deviations of   stationary Gaussian processes]
  {\sec On small deviations of   stationary Gaussian processes and related   analytic inequalities}
 
  \author{ Michel J.\,G. Weber 
}
\address{ Michel Weber: IRMA, Universit\'e
Louis-Pasteur et C.N.R.S.,   7  rue Ren\'e Descartes, 67084
Strasbourg Cedex, France. }
\email{michel.weber@math.unistra.fr
}
 \urladdr{http://www-irma.u-strasbg.fr/$\sim$weber/}
   
\begin{document}
 \maketitle
 \vskip 5pt 
    \begin{abstract}  Let $ \{ X_j, j\in \Z\}$ be a Gaussian stationary sequence having a spectral function $F$ of infinite type. 
Then
 for all $n$ and $z\ge 0$,$$  \P\Big\{ \sup_{j=1}^n |X_j|\le z \Big\}\le   \Big( 
\int_{-z/\sqrt{G(f)}}^{z/\sqrt{G(f)}} e^{-x^2/2}\frac{\dd x}{\sqrt{2\pi}} \Big)^n,$$
where $ G(f)$ is  the geometric mean of the Radon Nycodim derivative of the absolutely continuous part $f$ of $F$. The
proof uses 
    properties of finite Toeplitz forms. Let   
$ 
\{X(t), t\in
\R\}$ be a     sample continuous   stationary Gaussian  process with covariance function
$\g(u) $.  
  We also show that there exists an absolute constant $K$ such that for all $T>0$, $a>0$ with $T\ge \e(a)$,
$$\P\Big\{ \sup_{0\le s,t\le T} |X(s)-X(t)|\le  a\Big\}  \le \exp \Big \{-{ KT \over  \e(a) p(\e(a))}\Big\}
 ,$$ where  $\e (a)= \min\big\{ b>0: \d (b)\ge a\big\}$, $\d (b)=\min_{u\ge 1}\{ \sqrt{2(1-\g((ub))}, u\ge 1\}$,  and
 $   p(b) = 1+\sum_{j=2}^\infty 
{|2\g (jb )-\g ((j-1)b )-\g ((j+1)b )| \over 2(1-\g(b))}$. The proof is based on  some decoupling inequalities arising from Brascamp-Lieb
inequality. Both approaches are developed and compared on examples. Several other related results are established.
    \end{abstract}
\maketitle
      \vskip 3pt \noi  {\gssec 2000 AMS Mathematical Subject Classification}: {\phh Primary: 60F15, 60G50 ;
Secondary: 60F05.}  \par\noi  
{{\gssec Keywords and phrases}: {\phh small deviation, Gaussian process, stationary, decoupling coefficient, matrices with dominant
principal diagonal, Ger\v sgorin's disks, Toeplitz forms, eigenvalues, strong Szeg\"o limit theorem, geometric mean, Littlewood
hypothesis}.} 
  \vskip 5pt 
 
 

   
 \section{Introduction and Preliminary Results}   The  study of small deviations of  continuous  Gaussian processes and more general
continuous processes  is a very active domain of research.
This is also a  very specialized area,  rich of many specific results, mainly concerning typical processes having strongly regular
covariance structure, such as Brownian motion, Brownian sheet, fractional Brownian motions, integrated fractional Brownian motions, Hurst
processes, \ldots This aspect of the theory  has  naturally  many applications in statistics. It is also sometimes  related   to operator
theory.  

    The small   deviations problem for the class of stationary Gaussian processes is   of particular interest,  the way 
how      stationary and mixing properties interact  being notably not  quite well understood. This is the main focus of this work.
  Let   $X=  \{X(t), t\in \R\}$ be throughout  a     sample continuous   stationary Gaussian  process with covariance function
$\g(u)= \E X(t+u) X(t)$. The underlying   problem is the study for   small  $z $   and $T$ large,    $0<z<z_0$, $T_0\le T<\infty$  say, of
the probability
$$\P\big\{ \sup_{0\le s,t\le T} |X(s)-X(t)|\le  z\big\} .$$
   One can also separately consider asymptotics  for   $T\to \infty$,
$a$ being  fixed, or    $a\to 0 $, $T$  fixed.  
   The most celebrated example of stationary Gaussian  process is naturally the Ornstein-Uhlenbeck process
$  U(t)= W(e^t)e^{-t/2}$, $t\in
\R
$,
$W$ denoting the standard Brownian motion. And we know that for  $z>0$, there exist positive constants $K_1(z), K_2(z)$
such that for all $T\ge 1$ 
\begin{equation}\label{est/ou/cs}K_1(z) e^{-\l(z)T}  \le \P\Big\{\sup_{0\le s\le
T}|U(s)|<z\Big\}\le K_1(z)  e^{-\l(z)T} . 
\end{equation} 
Further $\l(z)\sim {\pi^2\over 4z^2}$   as $z\to 0$. See \cite{C}, Lemma 2.2. This precise estimate follows from  
earlier work  of Newell in which this question  is showed to be intimately linked to the Sturm-Liouville equation
\begin{equation} \label{sl/e} \psi''(x)-x\psi'(x)= -\l \psi(x), \qq \psi(-z)=\psi(z)=0.
\end{equation}
Let $\l_1\le \l_2\le \ldots$ and $\psi_1(x) , \psi_2(x),\ldots$ respectively denote the eigenvalues and normed eigenfunctions of
Eq.(\ref{sl/e}).   Here $\l_i, \psi_j$ depend on $z$ and it is known that $\psi_1, \psi_2, \ldots$  form an orthonormal sequence with
respect to the weight function
$e^{-x^2/2}$. And $\l(z)=\l_1$  in (\ref{est/ou/cs}).
According to \cite{N},
\begin{equation} \label{sl/n}\P\big\{\sup_{0\le s\le t}|U(s)|<z\big\}= {1\over (2\pi)^{1/2}}\sum_{k=1}^\infty
e^{-\l_k t}\Big(\int_{-z}^z\psi_k(x)e^{-x^2/2}\dd x\Big)^2.
\end{equation}
     For many purposes, the   weaker  estimate below  suffices,   and is moreover simpler to establish: 
    for    
$T\ge T_0$, $0\le z\le z_0$   
\begin{equation}\label{weaker/ou} e^{- K_1{ T \over z^2} }\le \P\Big\{ \sup_{0\le s,t\le T} |U(s)-U(t)|\le z\Big\} \le e^{- K_2{ T
\over z^2} } ,
\end{equation}
$K_1,K_2 $ being absolute constants.   The lower bound part follows from Talagrand's general lower bound in \cite{[T]}.
See \cite{AL},\cite{W1} for recent improvments.
  As to the upper bound part, it can   be for instance deduced from  Stolz's estimate \cite{S} (Corollary 1.2) or (\ref{uhlenb}).  
   The   small deviations problem of $X$   naturally     relies    on both the  behavior  of  $\g(u)$ near  0 and    near infinity.     
   At this regard, it is  worth observing that the (exponential) rate of decay of
$\g(u)$ near infinity is hidden in  (\ref{est/ou/cs}) and  (\ref{weaker/ou}).  Let us begin with the discrete case.
Let   $\uX=\{X_j, j\in \Z\}$ be a stationary Gaussian sequence. If the sequence $\uX$ is   i.i.d., then obviously
\begin{equation}\label{iid} \lim_{n\to \infty} \frac1{n}\log \P\big\{\sup_{j=1}^n |X_j|\le x\big\} =1, \qq \forall x>0.
\end{equation}
  
 It is rather unexpected   that this   holds   for a very large class of  stationary Gaussian sequences. It suffices in effect, that   
the geometric mean of the Radon-Nycodim derivative of the absolutely continuous part of its spectrum be finite; see  
Theorem \ref{szego} where a   more precise result is established.  Beyond this case,   that  question seems to loose  much
interest.  For instance if $\uX$ has absolutely continuous spectrum  with spectral density $f$, and $f$ has infinite geometric mean, then
$\uX$ is deterministic.  This yields extremely strong dependence between the successive variables $X_j$.   The condition that
$\sum_{n=1}^\infty |\E X_0X_n|<\infty$ is also sufficient for the validity of (\ref{iid}).
 \vskip 2pt 
 We  will study  
  these questions through essentially  two different ways: one is probabilistic, 
although based on a real analysis device, and the other  of spectral nature. 
We shall  also compare them on representative classes of examples. The first is the correlation approach, which is based on powerful
correlation inequalities   derived from  
  Brascamp--Lieb's inequality.  This is investigated in  Sections   \ref{sdec},\ref{section[corsup]},\ref{section[cgplh]}.
 We  notably establish for the continuous parameter case a rather general upper bound  integrating    the rate of decay of
$\g(u)$ near infinity. 
\vskip 1pt
 A first   relevant and little known    correlation estimate is   Gebelein's  inequality (\cite{BC},\cite{V}). Let $\nu$ be the centered
normalized Gauss measure on $\R$. Let $(U,V)$ be a Gaussian pair with $U\buildrel{\mathcal D}\over {=}V\buildrel{\mathcal D}\over {=}\nu$
and let $\rho= \E U V$. Then for any $f,h\in L^2(\nu)$
\begin{equation}\label{Gebel} |\E f(U)h(V) |\le |\rho| \|f\|_2\|h\|_2.
\end{equation}  
An analog result is Nelson's hyper-contractive  estimate, which  can be reformulated as follows
\begin{equation}\label{nelson} |\E f(U)h(V) |\le   \|f\|_p\|h\|_q,
\end{equation} 
where $(p-1)(q-1)\ge  \rho^2$. One can take in particular $p=q=1+|\rho|$.
We have given Guerra, Rosen and Simon formulation of Nelson's estimate \cite{GRS}, which was originally stated for the Ornstein-Uhlenbeck
process. They also established  for this process that
 \begin{equation}\label{grs} \Big|\E \prod_{j=1}^n f_j(U(ja))\Big |\le   \prod_{j=1}^n \|f_j(U(0))\|_p ,
\end{equation}
for all integers $n$, where $a>0$ and $p=(1-e^{-na})^{-1}(1+e^{-na})$.
A more general  form  was later proved in a  deep  work \cite{KLS} by   Klein, Landau and Shucker. See  Lemma \ref{decoupling}.
As already mentionned, the main ingredient is a real analysis inequality due to Brascamp--Lieb \cite{Bra}, which asserts that for
 any complex-valued functions $f_j$
and real numbers $1\le p_j\le\infty$, $j=1,\ldots k$ with $\sum_{j=1}^k \frac1{p_j}=n\ge k$, $n$ integer, if $f_j\in L^{p_j}(\R)$,
  then for any vectors $a^j$ in $\R^n$, $j=1,\ldots k$,
\begin{equation}\label{bl} \Big| \int_{\R^n}  \prod_{j=1}^k f_j(\langle a^j, x\rangle)\,  \dd x\Big |\le D  \prod_{j=1}^k \|f_j\|_{p_j} ,
\end{equation}
and the constant $D$ is computable explicitely (see \cite{Bra}, Theorems 1,5). Inequalities of this sort were     
intensively investigated in the recent years, see
\cite{Ba} for instance and references therein.

\vskip 2pt 
The second approach
   is based on the  theory of   finite Toeplitz forms,  especially   strong
Szeg\"o limit theorem and is investigated in Section \ref{toep.szeg}. We obtain   comparable upper and lower estimates under simple
conditions regarding the spectral density of the stationary Gaussian sequence. It seems by the way   rather evident to assert  that any
reasonable attempt for developing a small deviation theory of stationary Gaussian processes cannot be undertaken without including a
large account from the asymptotic theory of eigenvalues of finite Toeplitz forms. This can be well  illustrated as follows.
   Let $\uX$  having a 
spectral density function
$f(t)$ and put
$$ c_k=\frac1{2\pi}\int_{- \pi}^{ \pi} e^{ik t} f(t)   \dd t, \qq  \quad  k\in \Z . $$
  Let $\Gamma_n$ denote the covariance matrix   of  $(X_1, \ldots, X_n)$, obviously $\E X_jX_k= c_{j-k}$.  
 The study of the asymptotic distribution  of its eigenvalues, as $n$ tends to infinity, can be equivalently viewed as the one
of the finite Toeplitz forms
$$T_n(f)= \sum_{j,k=0}^n c_{j-k} a_j\bar a_k=\frac1{2\pi} \int_{-\pi}^\pi \Big|\sum_{ k=0}^n   a_ke^{ikt}\Big|^2f(t)\dd t, \qq
n=0,1,\ldots$$ This is an old question. Let $m$ and $M$ denote the essential  lower and upper bound $f$ respectively. Assume for instance
that  $0<m\le M<\infty$. Denote  by  
 $\l_1^n, \ldots, \l_{n+1}^n$,  
the eigenvalues of the Hermitian form $T_n(f)$, namely the roots of the
characteristic function
$T_n(f-\l)=0$.  As $\l_j^n\ge m>0$, it follows that $\det (\Gamma_n) >0$.    It is well-known   that the sets
$$\big\{\l_j^n\big\}\qq{\rm and} \qq \big\{f\big( -\pi +\frac{2j\pi}{n+2}\big)\big\}, \qq n\to \infty, $$
are equally distributed in the Weyl sense.   According to Szeg\"o's limit theorem (\cite{GS}, Chapter 5), for any
continuous function $F$     defined on
$[m,M]$,  
\begin{equation}\label{weyl}\lim_{n\to \infty} \frac{F(\l_1^n)+\ldots +F(\l_{n+1}^n)}{n+1}=\frac1{2\pi}\int_{- \pi}^{ \pi} F(f(t))  \dd t .
\end{equation}
     A well-known fact easily derived from   (\ref{weyl}) is that
 \begin{equation}\label{eig.weyl}\lim_{n\to \infty} \big[\det (\Gamma_n)\big]^{\frac1{n+1}}= \exp\Big\{ \frac1{2\pi}\int_{- \pi}^{ \pi}
\log f(t) \dd t\Big\}.
\end{equation}
Indeed, as $\det (\Gamma_n)=\l_1^n  \ldots  \l_{n+1}^n$, it suffices to apply (\ref{weyl})  with $F(\l) =\log \l$, $\l>0$. 
 This has immediate consequences concerning the  small values of $(X_1,\ldots, X_n)$, $n\to \infty$.
\vskip 2pt Finally we examine in Section \ref{section[ddp]}  the non-stationary case and use the  
  convenient notion of  matrices with dominant principal diagonal. This direction was   explored by Li and Shao (see \cite{LS2}, see also
the survey
\cite{LS} and the references therein, as well   the earlier work  of Marcus \cite{M}), and some improvments   of their results are
established.   We also clarify the relevance of this notion in the context  of eigenvalues of
Hermitian matrices by linking it with Ger\v sgorin's Theorem.  
 
\vskip 2pt We
believe
  that the used approaches are potentially    more developable and should certainly
allow  to improve on  the general knewledge of small deviations in the stationary case.
\vskip 5pt \noi {\gsec Basic Estimates.}   
Recall well-known Kathri-Sid\'ak's inequality   implying for any   Gaussian vector
$(X_1,\ldots, X_J)$ that
\begin{equation}\label{ks} \prod_{j=1}^J \P \{  |X_j|\le z \}\le \P\big\{ \sup_{j=1}^J |X_j|\le z\big\}.
\end{equation} 
 Now recall Boyd's precise estimate of  Mills'  ratio $R(x) =
e^{x^2/2}\int_x^\infty e^{-t^2/2}\dd t$:      for all $x\ge 0$,
\begin{equation}\label{basicsmill}  {\pi\over \sqrt{x^2+2\pi} +(\pi -1)x}\le R(x) \le
{\pi\over \sqrt{(\pi -2)^2x^2 +2\pi} +2x } . 
\end{equation}
 Notice that  
 both bounds tend to $ (\frac{\pi}{2} )^{1/2} $ as
$x$ tends to $0$.  
  Mill's ratio is clearly directly related to the Laplace transform of  $g$ since for any real $\l\ge 0$,
\begin{equation}\label{basicsmall}\E e^{-\l|g|}= \big(\frac{2}{\pi}\big)^{1/2} R(\l). 
\end{equation}
 It follows that $\E e^{-\l|g|}\sim  (\frac{2}{\pi} )^{1/2}\l^{-1}$, $\l\to \infty$. Further, for all $\l>0$
\begin{equation}\label{basiclpt}\E e^{-\l|g|}\le \min \big( \frac{\sqrt 2}{\l\sqrt \pi}, 1\big). 
\end{equation}
We refer for instance to \cite{W}   Section 10.1   for these facts and more details.  
\smallskip\par \noi {\gsec Notation--Convention.} The letter $g$ is used to denote  throughout   a standard Gaussian random variable. 
  All  Gaussian random variables, Gaussian sequences or processes  we consider are implicitely assumed  to be {\it centered}. Further,  $g_1, g_2,
\ldots$ will always denote a sequence of i.i.d. Gaussian standard random variables, and the Ornstein-Uhlenbeck process is denoted by
$  U(t), t\ge 0 $.  The notation $f(t)\asymp h(t)$   near $ t_0\in\overline \R$ means that for $t$ in a
neighborhood of $t_0$, $c |h(t)|\le  |f(t)|\le C |h(t)| $    for some constants  $0<c\le C <\infty$. Finally, we convince that
$\frac1{0}=\infty$.
 
 \section{Stationary Gaussian Processes with finite decoupling coefficient} 
\label{sdec} Let
$\{X(t), t\in \R\}$ be a stationary  
Gaussian process with continuous sample paths and let $\g(u)= \E X(0)X(u)$ denote its covariance function. 
 
\begin{theorem}\label{deco}  
 Assume that $\sum_{j=1}^\infty |\g(jb)|<\infty$,  for all $b>0$. 
   Then there exists an absolute constant $K$ such that for all $T>0$, $a>0$ with $T\ge \e(a)$,
$$\P\Big\{ \sup_{0\le s,t\le T} |X(s)-X(t)|\le  a\Big\}  \le \exp \Big \{-{ KT \over  \e(a) p(\e(a))}\Big\}
  ,$$ where  $\e (a)= \min\big\{ b>0: \d (b)\ge a\big\}$, $\d (b)=\min_{u\ge 1}\{ \sqrt{2(1-\g((ub))}, u\ge 1\}$  and
$$   p(b) =1+\sum_{j=2}^\infty 
{|2\g\big(jb\big)-\g\big((j-1)b\big)-\g\big((j+1)b\big)| \over 2(1-\g(b))} . $$  
 \end{theorem} 
\begin{remark}  In the case of the Ornstein-Uhlenbeck process, it can be shown that $p(b)$ tends to some positive finite limit as $b$
tends to $0$. Indeed,
$$ p(b)=1+{ \big| 2  -e^{ b/2}-e^{- b/2}\big| \over 2(1-e^{- b/2})}\sum_{j=2}^\infty e^{-jb/2}=1+e^{- b }{ \big| (1-e^{ b/2})+(1-e^{-
b/2})\big|
\over 2(1-e^{- b/2})^2}   
  . $$ 
By developing near $b=0$, we have 
\begin{eqnarray*} (1-e^{ b/2})+(1-e^{- b/2})&=& (1-[1+{b\over 2} +{1\over 2}{b^2\over 4} ])+(1-[1-{b\over 2} +{1\over 2}{b^2\over 4} ])+
\mathcal O(b^3)
 \cr &=&     - {b^2\over 4}  + \mathcal O(b^3),
\end{eqnarray*}
so that 
$$ p(b) \sim 1+e^{- b }{   b^2  
\over 8(1-e^{- b/2})^2}   
  \sim 1+e^{- b }{   b^2  
\over 8(b^2/4)}\sim {3\over 2},  \qq\qq b\to 0. $$
  Moreover $\d(b)=  \sqrt{2(1-e^{-  b/2})}\sim \sqrt{b}$ as $b\to 0$.    Theorem \ref{deco} thus implies the upper bound part of
(\ref{weaker/ou}). 
\end{remark} 
We begin with   recalling   some    decoupling
inequalities    (\cite{KLS}, Theorems 1 and 2) due to  Klein, Landau and Shucker, and which  turn  up to be not so known.  
\begin{lemma} \label{decoupling} a) Let $X=\{X_t, t\in \Z^d\}$ be a stationary Gaussian process  with finite decoupling coefficient
$p $, that is: 
\begin{equation}\label{p0} p  =\sum_{k=1}^\infty  {|\E X_0 X_k| \over \E  X_0^2} <\infty. 
\end{equation} 
  Let
$\{f_k,k\ge 1\}$  be a sequence of complex-valued  measurable
functions. Then for each finite subset
$J$ of
$\Z^d$,  
 $$\Bigl|\E  \prod_{j\in J}f_j(X_j) \Bigr|\le \prod_{j\in
J}\bigl\|f_j(X_0)\bigr \|_{p } . $$
b)  Let $\{X_t, t\in \R^d\}$ be a stationary Gaussian process, continuous in mean, with Riemann approximable covariance
function. Let $V$ be a $\C$-valued measurable function of a real variable. Assume that $V(X_0)$ is integrable. Then, for all
bounded measurable subsets $B$ of $\R^d$, 
$$\Big|\E  \exp\Big\{\int_B V(X_t) \dd t\Big\}\Big|\le \big\|\exp\big\{V(X_0)\big\}\big\|_p^{|B|},  $$
where 
\begin{equation}\label{p1}
p= \int_{\R^d}\frac{\E (X_0X_t)}{\E X_0^2}\dd t<\infty,
\end{equation}
  and $|B|$ denotes the Lebesgue measure   of $B$.\end{lemma} 
  In either case, the proof   relies on   inequality  (\ref{bl}). It is of matter to briefly explain its principle. At first,   a
similar result (see Lemma \ref{cyclic0}) is established for   cyclic stationary Gaussian processes. The proof is next achieved by
  approximating
$X$ with cyclic stationary Gaussian processes. A key observation is then that
$$r_N(n)=\sum_{k\in \Z^d} r(n+kN),  \qq \qquad r(u)=\E X_0X_u,  $$
is, under   condition (\ref{p0}), an $N$-periodic covariance function, and $\lim_{N\to \infty}r_N(n)= r (n)$ for
all $n$, which is a remarkable fact.  The proof for the continuous parameter case is similar.
 
\begin{proof}[Proof of Theorem \ref{deco}]
Notice that for each fixed real    $b>0$, the Gaussian sequence
$$\xi_b(j)= X(jb)-X((j-1)b), \qq\qq j=1,2,\ldots $$ is   stationary.
    Let indeed $\ell,u\ge 1$,  
then
\begin{eqnarray}\label{corxi}\E \xi_b(\ell)\xi_b(\ell+u) &=& \E \big(X((\ell+u)b)-X(( \ell+u
-1)b)\big)\big(X(\ell b)-X((\ell-1)b)\big) 
\cr&= &2\g\big(ub\big)-\g\big((u-1)b\big)-\g\big((u+1)b\big),
\end{eqnarray}
which   only depends on $u$.   It has finite decoupling coefficient, and more
precisely  
$$  \sum_{j=1}^\infty  {|\E \xi_b(1) \xi_b(j)| \over \E  \xi_b(1)^2}=1+\sum_{j=2}^\infty 
{|2\g\big(jb\big)-\g\big((j-1)b\big)-\g\big((j+1)b\big)| \over 2(1-\g(b))}= p(b)  <\infty. $$ 
Further 
 if $F$ denotes the spectral function of $X$, $\gamma(u)=\int_\R e^{iu\l}F(d\l)$, then 
\begin{eqnarray}\label{Fcorinc}\E \xi_b(\ell)\xi_b(\ell+u)=\int_\R e^{-i\l u b}|e^{ib\l}-1|^2F(d\l).\end{eqnarray}
And 
$$ \sum_{j=1}^\infty  {|\E \xi_b(1) \xi_b(j)| \over \E  \xi_b(1)^2}=1+ \frac{\int_\R \sum_{j=2}^\infty
e^{-i\l u b}|e^{ib\l}-1|^2F(d\l)}{\int_\R  |e^{ib\l}-1|^2F(d\l)}
 .$$ 

 Let $T\ge b$.   Consider on
$[0,T]$ the subdivision $t_j= jb$,
$0\le j\le   n:=\lfloor T/b\rfloor$.  
 We have   
$$\|X((j+u)b)-X(jb)\|_2^2=2(1-\g((ub)) \ge 2\min_{u\ge 1}(1-\g((ub)) =\d^2 (b)  .  $$ 
Let $c=\sqrt{2/\pi}$. Let $a>0$ and choose $b$ so that 
 $   \d (b)\ge a$. 
Let $g$ denote a Gaussian standard random variable. By Lemma \ref{decoupling},
\begin{eqnarray*} \P\big\{\max_{1\le j\le n} |\xi_b(j)|\le  a\big\}&= &\E \prod_{i=1}^n\chi_{[-a,a]}(\xi_b(j))
 \le     \Big(\prod_{i=1}^n \P\big\{  |\xi_b(j)|\le
a\big\}\Big)^{1\over p(b)}
\cr &\le  & \P\big\{  |g|\le
{a\over \d (b)}\big\}\Big)^{n\over p(b)}
=\Big(\sqrt{2\over
\pi}\int_0^{ {a\over \d (b)} } e^{-x^2/2}dx\Big)^{n\over p(b)}
\cr &\le  & c ^{n\over p(b)}=e^{-{\lfloor T/b\rfloor\over p(b)}\log {1\over c}} \le      e^{-  {T \over
2p(b)b}\log {1\over c} } .
\end{eqnarray*}
 As 
$$\P\Big\{ \sup_{0\le s,t\le T} |U(s)-U(t)|\le a \Big\}  \le \P\Big\{\max_{1\le j\le n} |\xi_b(j)|\le  a\Big\}$$
by taking $b=\e(a)$, we obtain 
$$\P\Big\{ \sup_{0\le s,t\le T} |U(s)-U(t)|\le  a\Big\}  \le e^{-K  { T \over \e(a) p(\e(a))}   } ,$$
with $K={   1\over 2  }\log {1\over c}={   1\over 4  }\log {\pi\over 2}$.
\end{proof}
\begin{remark} \rm A  direct application of
the decoupling inequality to the  sequence $X(jb)$ instead of $  X(jb)-X((j-1)b)$ only provides a bound with a decoupling coefficient
which may tend  to infinity when $b\to 0$. So is in particular the case when $X$ is  the Ornstein-Uhlenbeck process.
\end{remark}

We also establish the following general upper bound.  
\begin{theorem} \label{supdec}Let    $\{X_t, t\in \R^d\}$ be a stationary sample continuous Gaussian process.   Assume that  condition
(\ref{p1}) is fulfilled.  
 For any $z>0$, any bounded  interval $B$ of $\R^d$, 

$$ \P\Big\{     \sup_{ t\in B  }  |X_t| \le  z  \Big\}\le 
  \big( e^{ p}   \, \P\{|g|< z\}  \big)^{\frac{|B|}{p}}   .
$$
\end{theorem} 
   \begin{proof} Let $f:\R\to \C$ be measurable,  such that $\E |f(X_0)|<\infty$,  and let $  \l,\theta$ be positive reals. Applying part
b) of  Lemma
\ref{decoupling}   with $V(x)=  -\l  f(x)  
$ gives   \begin{eqnarray*}
 \E  \exp\Big\{- \l  \int_{ B  }  f(X_t)    \dd t\Big\}
  \le \big\|e^{-     \l  f(X_0)    }\big\|_p^{|B|}  
 =  \big(\E e^{-  p\l    f(g) }\big)^{\frac{|B|}{p}}  .
\end{eqnarray*} 
 Thereby, 
\begin{eqnarray*} \label{int}
\P\Big\{     \int_{ B  }  f(X_t)    \dd t \le \theta  \Big\}
&=&\P\Big\{ -
 \l\int_{ B  }  f(X_t)    \dd t\ge  -\l \theta \Big\}
\cr &\le & \min \Big(e^{\l\theta }\, \E  \exp\Big\{-\l  \int_{ B  }  f(X_t)    \dd t\Big\}, 1\Big)
\cr &\le & 
 \min\Big(e^{\l \theta}  \big(\E e^{-  p\l    f(g) }\big)^{\frac{|B|}{p}},1\Big)  .
\end{eqnarray*}
Apply this  to $f(x)=|x|^r, 0<r<\infty$.  Put 
\begin{equation}\label{notation} \|X\|_{r, B}= \Big(\frac1{|B|}\int_{ B  }  |X_t|^r    \dd t  \Big)^{\frac1{r}}, \qq 
\quad  \|X\|_{\infty, B}= 
\sup_{ t\in B  }  |X_t| . 
\end{equation}
 Notice first that $\|X\|_{\infty, B}=\lim_{r\to \infty}\|X\|_{r, B}$, almost surely, since $X$ is sample continuous. Take  $\theta= z^r
|B|$.  This gives
\begin{eqnarray*} 
\P\big\{     \|X\|_{r, B} \le z  \big\}
  &\le & 
 \inf_{\l >0}\, e^{\l z^r |B|}  \big(\E e^{-  p\l    |g|^r }\big)^{\frac{|B|}{p}}   .
\end{eqnarray*} 
Choose now $\l= z^{-r}$. Then 
\begin{eqnarray*} 
\P\big\{     \|X\|_{r, B} \le z  \big\}
  &\le & 
  \big( e^{ p}   \, \E e^{-  p    \frac{|g|^r}{z^r}}  \big)^{\frac{|B|}{p}}   .
\end{eqnarray*} 
But
\begin{eqnarray}\label{disap}\qq \lim_{r\to\infty}  e^{-  p    \frac{|g|^r}{z^r}} \buildrel{a.s.}\over{=}\begin{cases} 1, & \qq{\rm
if}\qq |g|< z\cr 0, & \qq{\rm if}\qq |g|> z. 
\end{cases}\end{eqnarray} 
Thus $p$ disappears from the limit. By using the dominated convergence theorem, we  get
$$\lim_{r\to\infty} \E e^{-  p    \frac{|g|^r}{z^r}} =  \P\{|g|< z\}. $$
Hence,  
$$ \P\big\{     \|X\|_{\infty, B} \le  z  \big\}\le \liminf_{r\to \infty} \P\big\{     \|X\|_{r, B} \le z  \big\}
   \le  
  \big( e^{ p}   \, \P\{|g|< z\}  \big)^{\frac{|B|}{p}}   .
$$
\end{proof}

\begin{remark} ({\gsec Ergodic maximal equality}) Introduce the ergodic maximal function  
 $$   \mathbb M^*(X)= \sup_{T>0} M_T(X) \qq{\rm where} \qq \mathbb M_T(X)= \frac{1}{T}\int_0^{T}|X_t|\dd t  .$$
As a special case of   a fine result from ergodic theory, namely
Marcus-Petersen's maximal equality for ergodic flows (\cite{W}, p.133), we have 
\begin{equation}\label{mp}\P\big\{  \mathbb M_\infty(X) \le \a\big\}= 0,
\end{equation} 
if $\a< \sqrt{ {2}/{\pi}  }$. 
A slightly less precise result can be directly derived from the first part of the above proof, in which only assumptions of Lemma
\ref{decoupling}, part b) are used. A simple modification of this one, also yields for all
$\theta >0$, $B$ with
$|B|>0$,  
\begin{equation}\label{int1} \P\Big\{  \frac{1}{|B|}\int_{B}|X_t|  \dd t\le \theta\Big\}\le 
 \min \big( { e  \big(  {  2}/{  \pi}\big)^{1/2}  \theta } , 1 \big)^{\frac{|B|}{p}} . 
\end{equation}
 Indeed, using  
   (\ref{basiclpt})  we have  with $c=\big(  {  2}/{  \pi}\big)^{1/2}$,
 \begin{eqnarray*} \P\Big\{  \int_{B}|X_t|  \dd t\le
z\Big\}&=&\P\Big\{ - \l \int_{B}|X_t| 
\dd t\ge  -\l z\Big\}
\cr &\le & \min\Big( e^{\l z}\, \E  \exp\Big\{- \l  \int_{B}|X_t| \dd t\Big\}, 1\Big)
\cr &\le &   \min\Big( e^{\l z}  \big( \frac{c}{\l p  } \big)^{\frac{|B|}{p}},1\Big) .
\end{eqnarray*}
Letting $z=\theta |B|$, $\l =1/p\theta$,   we deduce 
$$ \P\Big\{  \frac{1}{|B|} \int_{B}|X_t|  \dd t\le z\Big\}\le 
 \min \big( { e c  \theta } , 1 \big)^{\frac{|B|}{p}}. $$
By taking $B=[0,T]$, it follows that     for all $\theta<\frac{\sqrt{
{\pi}/{2}}}{e }$, ($e$ being the Neper number)
  $$\P\big\{  \mathbb M_\infty(X) \le
 \theta\big\}\le \limsup_{T\to \infty} \P\Big\{  \frac{1}{T}\int_{0}^T|X_t|  \dd t\le
 \theta\Big\}\le 
 \limsup_{T\to \infty}  (  e  \sqrt{  2/  \pi}  \theta )^{\frac{T}{p}}=0.$$
   As $ 2e> \pi $, this is slightly less precise than (\ref{mp}).    
\end{remark}    
 \section{Correlated Suprema} 
 \label{section[corsup]}  Consider now the similar question  for correlated suprema. Let $I_1,
\ldots, I_J$ be   bounded, pairwise disjoint intervals, and associate to them the sets
$$C_j(X) = \big\{\sup_{ t\in I_j} | X(t)|\le z_j\big\}, \qq j=1,\ldots, J $$
where $z_j$ are positive reals.    
  By H\"older's inequality,
$$\P\Big\{ \bigcap_{j=1}^J C_j(X)\Big\}\le C\prod_{j=1}^J\P \{   C_j (X)\}^\s , \qq \s={1\over J}. $$
In general that inequality cannot be improved. In particular there is no reason for $\s$ to be independent of $J$.  However 
    when $X=U$, namely for the Ornstein-Uhlenbeck process, this can be much improved.
\begin{proposition}\label{p2} For any    pairwise disjoint bounded intervals $I_1,
\ldots, I_J$, any positive reals $z_j$,    
\begin{eqnarray*}  \prod_{j=1}^J\P    \big\{\sup_{
t\in I_j} | U(t)|\le z_j\big\} \le \P\Big\{ \bigcap_{j=1}^J \big\{\sup_{ t\in I_j} | U(t)|\le z_j\big\} \Big\}   \le   
\prod_{j=1}^J\P    \big\{\sup_{ t\in I_j} | U(t)|\le z_j\big\}^{1\over p } , 
  \end{eqnarray*}  
where 
$$p=  \frac{ 1+e^{-|I_1|-\ldots -|I_J| } }{
1-e^{-|I_1|-\ldots -|I_J|}  }. $$\end{proposition}
\begin{proof}  
 Let $N>0$ be some large   integer. Since $I_j$ are bounded,  we have 
$$\P \{   C_j(U) \}= \lim_{N\to \infty}\P \{C_{j,N}(U) \}\qq{\rm where} \quad C_{j,N}(U)=\big\{\sup_{\frac{\ell}{N}    \in I_j }
|U(\frac{\ell}{N} )|\le z_j\big\}.$$ 
The first inequality follows   by proceeding by approximation and using inequality (\ref{ks}). Let $\nu_N= \#\{\ell 
:\frac{\ell}{N}   
\in
\cup_{j=1}^J I_j  
\}
$. By using   (\ref{grs}), we have
\begin{eqnarray*}\P\Big\{ \bigcap_{j=1}^J C_{j,N}(U)\Big\}&= &
\E \prod_{j=1}^J\prod_{\frac{\ell}{N}    \in I_j }\chi\big\{|U(\frac{\ell}{N})|\le z_j
\big\}
\cr &\le &
\prod_{j=1}^J\prod_{\frac{\ell}{N}    \in I_j }\P\big\{|U(\frac{\ell}{N})|\le z_j \big\}^{p_N}
\cr \hbox{(by  (\ref{ks}))}\qq &\le &
\prod_{j=1}^J \P\big\{\sup_{\frac{\ell}{N}    \in I_j }|U(\frac{\ell}{N})|\le z_j \big\}^{p_N}
 =   \prod_{j=1}^J\P \{  C_{j,N}(U)
\}^{p_N}   , 
\end{eqnarray*}  
where 
$$ p_N=\frac{ 1+e^{-\nu_N/N } }{
1-e^{-\nu_N/N}  } .  
 $$  But
$$\lim_{N\to \infty}\frac{\nu_N}{N}= |I_1|+\ldots +|I_J|. $$
Therefore $p_N\to p$ with $N$.     Letting $N $ tend to infinity in   the above inequality achieves the proof.  
 \end{proof}  

Now let  $I_j= n_j +I$ where $I$ is some fixed bounded interval and $n_j\uparrow \infty$ with $j$ and such that
$n_{j+1}-n_j\ge |I|$, $j\ge 1$. Put 
\begin{equation}\label{MI} M(I, n_1,\ldots, n_J)=\sup_{t\in I, \atop 1\le  j\le J}| U(t+n_j)|
\end{equation}
\begin{theorem} {\sl (Existence of the Limit)} \label{c2}  For   $z>0$,
 $$\lim_{J\to \infty}{\log   \P\{ M(I, n_1,\ldots, n_J)\le z\} \over J }= \log \P\big\{ \sup_{t\in I}| U(t)|\le z\big\}.$$ 
\end{theorem} 
\begin{proof} We have from Proposition \ref{p2}, for $J\ge 1$,
$$\P\big\{ \sup_{t\in I}| U(t)|\le z\big\}^{J }\le \P\Big\{ M(I, n_1,\ldots, n_J)\le z\Big\}\le \P\big\{ \sup_{t\in I}| U(t)|\le
z\big\}^{J/p_J}, $$ where $p_J=\frac{1+e^{-J|I|}}{1-e^{-J|I| }}$. Taking logarithms  and using the fact that   $p_J\to 1$ with $J$ gives
the result.
\end{proof}
One can also establish that 
\begin{corollary} \label{c3}   For $z>0$,
$$\lim_{J\to \infty}{\log   \P\{ M(I, n_1,\ldots, n_J)\le z\} \over J }=\inf_{J\ge 1}{\log   \P\{
M(I, n_1,\ldots, n_J)\le z\}\over J },
$$
where $M(I, n_1,\ldots, n_J)$ is defined in (\ref{MI}). 
\end{corollary}
Introduce a notion. Let $\uc=\{c_n, n\ge 1\}$ be positive reals tending to $c\ge 1$. We say that a sequence    $\{\p_n ,n\ge 1\}$ of
real numbers is
 $\uc$-subadditive,     if  
 $$ \p_{n_1+\ldots +n_k}\le c_{n_1+\ldots +n_k}(\p_{n_1 }+\ldots+ \p_{n_k}) $$  for all integers $n_1,\ldots , n_k$, $k\ge 1$.
 \begin{lemma}\label{csub} {\sl (Extended Subadditive Lemma)}
If $\{\p_n ,n\ge 1\}$ is a $\uc$-subadditive sequence of real numbers, then
$$
\inf_{n\ge 1}{\p_n\over n }\le \liminf_{n\to \infty}{\p_n\over n }\le \limsup_{n\to \infty}{\p_n\over n }
\le c^2\inf_{n\ge 1}{\p_n\over n }  . $$\end{lemma}
When $c_n\equiv 1$, this is a well-known device having many applications, in ergodic theory notably.  
\begin{proof} It is a simple modification of the classical proof of the case $c_n\equiv 1$. Fix an arbitrary positive integer $N$ and
write
$n= j_n N + r_n$ with $1\le r_n\le N$.  Then,
\begin{eqnarray*}
\inf_{n\ge 1}{\p_n\over n }\le {\p_n\over n }&\le& c_{j_n N + r_n}\, {\p_{j_n N}+ \p_{r_n  }\over n } \le c_{j_n N + r_n} \,{\p_{j_n
N}\over j_n N } + c_{j_n N + r_n}\,{\p_{r_n  }\over  n   }
\cr&\le & c_{j_n N + r_n}c_{j_n N }\,{j_n\p_{  N}\over j_n N } + c_{j_n N + r_n}\,{\p_{r_n  }\over  n   } 
\cr &\le &c_{j_n N + r_n}c_{j_n N }\,{ \p_{
N}\over   N } + c_{j_n N + r_n}\big(\max_{r\le N}|\p_r|\big)/  n   . 
\end{eqnarray*} When $n$ tends to
infinity, we have that ${j_n\over n}\to  \frac1N$. As $c_{j_n N + r_n}.c_{j_n N }\to c^2$, we get
 $$
\inf_{n\ge 1}{\p_n\over n }\le \liminf_{n\to \infty}{\p_n\over n }\le \limsup_{n\to \infty}{\p_n\over n }
\le c^2{ \p_{  N}\over   N }  . $$
Since  $N$ was arbitrary, the lemma     is proved.
\end{proof}
 
 \begin{proof}[Proof of Corollary \ref{c3}] Apply this to $\p_J=\log \P\{M(I, n_1,\ldots, n_J)\le z\}$. By Corollary \ref{c2}  and
stationarity,  
\begin{eqnarray} \p_{J+K}&= &\log \P\{\sup_{j\le J+K}\sup_{t\in I} |U(n_j+t)|\le z\}\le {1\over p_{J+K}}\log \P\{ \sup_{t\in I}
|U(t)|\le z\}^{J+K}
\cr &= & {1\over p_{J+K}}\log \Big(\prod_{j\le J }\P\{  \sup_{t\in I} |U(n_j+t)|\le z\} \cdot \prod_{j\le  K}\P\{  \sup_{t\in I}
|U(n_j+t)|\le z\}\Big)
\cr &\le  & {1\over p_{J+K}}\log   \P\Big\{  \sup_{j\le J  }\sup_{t\in I} |U(n_j+t)|\le z\Big\} \P\Big\{ \sup_{j\le  K} \sup_{t\in I}
|U(n_j+t)|\le z\Big\} 
\cr &= & {1\over p_{J+K}}(\p_{J }+\p_{ K})  .
\end{eqnarray}
But   $p_J=\frac{1+e^{-J|I|}}{1-e^{-J|I| }}$.
Similarly, $\p_{J_1+\ldots + J_s}\le {1\over p_{J_1+\ldots + J_s}}(\p_{J_1 }+\ldots +\p_{ J_s})$. 
Thus $\{g_n,n\ge 1\}$ is $\uc$-subadditive with  $\uc=\{p_J, J\ge 1\} $. Now $p_J=\frac{1+e^{-J|I|}}{1-e^{-J|I| }}\to 1$ as $J$ tends to
infinity. By Lemma
\ref{csub}, we deduce that 
\begin{eqnarray*}
\inf_{J\ge 1}{\log \P\{M(I, n_1,\ldots, n_J)\le z\}\over J }&\le &\liminf_{J\to \infty}{\log \P\{ M(I, n_1,\ldots, n_J)\le z\}\over J
}\cr &\le& \limsup_{J\to \infty}{\log \P\{ M(I, n_1,\ldots, n_J)\le z\}\over J }
\cr &\le&  \inf_{J\ge 1}{\log \P\{M(I, n_1,\ldots, n_J)\le z\}\over J } . 
\end{eqnarray*}
\end{proof}
\section{Cyclic   Gaussian Processes}   
  \label{section[cgplh]} As mentionned in Section \ref{sdec}, these processes played   a
key role in  \cite{KLS}. The following lemma, which we state for our need is the crux of the proof of Lemma
\ref{decoupling}. Although it is valid for cyclic stationary Gaussian processes $\{X_t, t\in \R^d\}$ with an arbitrary period
$(b_1,\ldots, b_d)$, we state it in the standard case of period $( 1,\ldots, 1)$, namely with fundamental index   
$\T^d=\R^d/\Z^d=[0,1[^d$.  
\begin{lemma}{\rm (\cite{KLS}, Theorem 3)}\label{cyclic0} Let $\{X_t, t\in \R^d\}$ be a 1-periodic stationary Gaussian process, continuous
in quadratic mean.  Let $V:
\R\to\C$ be measurable and such that $V(X_0)$ is integrable. Then, for all measurable subsets $B$ of   $\T^d$, 
$$ \Big|\E \Big( \exp\Big\{ \int_B V(X_t)\Big\} \Big)\Big|\le \big\|\exp(V(X_0)) \big\|_p^{|B|} $$
where 
$$p = \int_0^1   \frac{\big|\E X_0X_t\big|}{ \E X_0^2 } \dd t.$$
\end{lemma}
\begin{proposition} \label{cyclic} Let $Y_t=
\sum_{n=1}^N a_n\big(  g^1_n
\cos 2\pi nt +g^2_n
\sin 2\pi nt
\big)
$,
$t\in
\T=\R/\Z=[0,1[$, where $a_n$ are reals and 
$g^1_n, g^2_n$ are  mutually independent Gaussian standard random variables. Let $s^2=\sum_{n=1}^N a_n^2 $. For $\theta>0$ and
$B\subset\T$ interval of length $|B|$,
 $$ \P\Big\{   \sup_{t\in B  }  |Y_t|   \le \theta  \Big\} 
   \le  
  \Big( e^{ p}   \, \P\Big\{| g|< \frac{\theta}{s}\Big\}  \Big)^{\frac{|B|}{p}}   ,$$
where 
$$p = \int_0^1   \frac{\big|\sum_{n=1}^N a_n^2\cos 2\pi nt\big|}{ \sum_{n=1}^N
a_n^2 } \dd t.$$
\end{proposition} 
 \begin{proof}
Notice that $\E Y_sY_t=\sum_{n=1}^N a_n^2\cos 2\pi n(s-t)$. The  proof is very similar
to that  of Theorem
\ref{supdec}, except that we have a different decoupling coefficient: 
$$p= \int_0^1  \frac{|\E Y_0Y_t|}{\E Y_0^2} \dd t= \int_0^1   \frac{\big|\sum_{n=1}^N a_n^2\cos 2\pi nt\big|}{ \sum_{n=1}^N
a_n^2 } \dd t.$$
We only indicate the  necessary changes. The proof is identical with $Y_t$ in place of $X_t$ until  (\ref{int}), where there is a
slight  modification  due to the fact that
$Y_0 = \sum_{n=1}^N a_n   g^1_n\buildrel{\mathcal D}\over {=}sg $,  
($s^2=\sum_{n=1}^N a_n^2 $). Using Tchebycheff's inequality and Lemma \ref{cyclic0}, gives 
\begin{eqnarray} \label{intY}
\P\Big\{     \int_{ B  }  f(Y_t)    \dd t \le \theta  \Big\}
  &\le & \min \Big(e^{\l\theta }\, \E  \exp\Big\{-\l  \int_{ B  }  f(Y_t)    \dd t\Big\}, 1\Big)
\cr &\le & 
 \min\Big(e^{\l \theta}  \big(\E e^{-  p\l    f(sg) }\big)^{\frac{|B|}{p}},1\Big)  .
\end{eqnarray}
Applying this  with $f(x)=|x|^r$,  $\theta= z^r |B|$, $\l= z^{-r}$   gives in exactly the same manner, with the notation  
(\ref{notation}),
  \begin{eqnarray*} 
\P\big\{     \|Y\|_{r, B} \le z  \big\}
  &\le & 
  \big( e^{ p}   \, \E e^{-  p   ( \frac{ s|g| }{z })^r}  \big)^{\frac{|B|}{p}}   .
\end{eqnarray*} 
 Hence,  by using (\ref{disap}),
$$ \P\big\{   \sup_{t\in B  }  |Y_t|   \le  z  \big\}\le \liminf_{r\to \infty} \P\big\{    \|Y\|_{r, B} \le z  \big\}
   \le  
  \big( e^{ p}   \, \P\{|sg|< z\}  \big)^{\frac{|B|}{p}}   .
$$
\end{proof}
An immediate consequence of Proposition \ref{cyclic} is that 
\begin{corollary}   \label{cyclrev} With the notation from Proposition \ref{cyclic}, for $z>0$,
$$\int_0^1    {\big|\sum_{n=1}^N a_n^2\cos 2\pi nt\big|}  \dd t\ge \Big({ \sum_{n=1}^N
a_n^2 }\Big)\ \frac{\log \frac{1}{ \P\{| g|< \frac{z}{s}\} }}{\log \frac{e}{
\P  \{  \|Y\|_{\infty, B}   \le  z 
 \}^{\frac{1}{|B|}}} }.
$$
\end{corollary}
 \begin{remark}{\gsec (Littlewood hypothesis)} Let $n_1<n_2<\ldots  $ be integers.  
 Consider the (generalized) Lebesgue    constants
$$\t_N=\int_0^1    
\big|\sum_{k=1}^N
e^{2i \pi n_kt}\big|  
\dd t,  \qq\quad  N=1,2,\ldots$$
  Littlewood hypothesis (\cite {Ol} p.12 for instance) essentially concerns the behavior of Lebesgue
constants of   arbitrary ordered trigonometric systems, and can be formulated as follows: for any increasing sequence of integers,
$$   {\t_N}\ge C {\log N} , $$
where $C>0$ is an absolute constant. This was proved independently by Konyagin \cite{Ko} and McGehee, Pigno and Smith \cite{MPS} in
1981.    Consideration of the Dirichlet kernel shows that the above lower bound is best possible. See \cite{Z} p. 67.
\end{remark}\vskip 3pt
We shall deduce from Corollary \ref{cyclrev} 
 \begin{corollary}\label{cycl} For all positive integers   $N$, all $z>0$ and $B\subset \T$ interval,
$$\t_N\ge N|B|\, \Bigg(\frac{\log \frac{1}{\displaystyle{ \P\{| g|< z\}} }}{\log \frac{e^{|B|}}{
\displaystyle{\P  \Big\{    \displaystyle{ \sup_{t\in B}} 
\frac{1}{\sqrt N}\big|\sum_{1\le k\le N}   (  g^1_{ k}
\cos 2\pi n_kt +g^2_k\sin 2\pi n_kt
 ) 
\big|   \le z  
 \Big\}
} }}\Bigg).
$$
\end{corollary}
\begin{proof}
  Apply Corollary \ref{cyclrev}
 with the choice $a_n= 1/\sqrt N$, if $n=n_k$ for
some $k\le
 N$, and equal to
 $0$ otherwise.   We deduce 
$$\t_N\ge N\ \frac{\log \frac{1}{ \P\{| g|< z\} }}{\log \frac{e}{
\P  \Big\{    \displaystyle{ \sup_{t\in B}} \frac{1}{\sqrt N}
\Big|\sum_{1\le k\le N}   (  g^1_{ k}
\cos 2\pi n_kt +g^2_k\sin 2\pi n_kt
 ) 
\Big|   \le z  
 \Big\}^{\frac{1}{|B|}}} } 
$$
as claimed.\end{proof}

The above link  between $L^1$-norms of   trigonometric sums and Gaussian random trigonometric sums,  seems unexpected. This 
  suggests to examine it  more closely    using   results in
\cite{W1},\cite{W3}. This question will be investigated   elsewhere.    We conclude with a    remarkable example in which Anderson's
inequality is used and Talagrand's well-known lower bound since the corresponding entropy numbers are  very simple.  
 \begin{corollary} There exists an absolute constant $C$ such that for any  set of integers $J$,
$$   \int_0^1    \big|\sum_{n\in  J}   \frac{1}{n^2} \cos   nt\big|   \dd t\ge  
   C\big(\sum_{n\in  J}   \frac{1}{  n^2}\big)^{2} .  $$
\end{corollary}

\begin{proof} Let  
$$X_t=
\sum_{n=1}^\infty   \frac{1}{  n} \big(  g^1_k
\cos nt +g^2_k
\sin nt
\big), \qq Y_t=
\sum_{n\in  J}   \frac{1}{  n} \big(  g^1_n
\cos   nt +g^2_n
\sin   nt \big)
$$
 Then $\E X_s^2=\sum_{n=1}^\infty  \frac{1}{  n^2}= \frac{\pi^2}{6}$ and 
$$   \E X_s X_t= \sum_{n=1}^\infty  \frac{\cos n (s-t)}{  n^2} = 
\frac{3|s-t|^2-6\pi |s-t|+2\pi^2}{  12}.$$
Thus   $d^2(s,t)= \E( X_s- X_t)^2= \pi|s-t| -\frac{1}{2}|s-t|^2\sim  \pi|s-t|$ as $|s-t|\to 0$.  It follows that $N([0,1], d, \e)\asymp
\e^{-2}$.  By using Talagrand's lower bound (see  \cite{[T]}),
$$ \P\Big\{\sup_{0\le t\le 1}|X_t|\le \e\Big\}\ge e^{-K\e^{-2}}. $$
 Now since $Y$ and $X-Y$ are independent,
 by using Anderson's inequality, we get $$ \P\Big\{\sup_{0\le t\le 1}|X_t|\le \e\Big\}\le  \P\Big\{\sup_{0\le t\le 1}|Y_t|\le \e\Big\}. $$
Therefore 
$$\P\Big\{\sup_{0\le t\le 1}|Y_t|\le \e\Big\}\ge e^{-K\e^{-2}}. $$
We have $Y_0=\sum_{n\in  J}   \frac{1}{n}g^1_n$  and   $s(J)=\big(\sum_{n\in  J}   \frac{1}{  n^2}\big)^{1/2} $, 
$$p =   s(J)^{-2}\int_0^1  {\big|\sum_{n\in  J}   \frac{1}{n^2} \cos   nt\big|}  \dd t.$$  
Applying Proposition \ref{cyclic} with $B=[0,1]$ gives, 
    $$ e^{-K\theta^{-2}} \le \P\big\{   \sup_{0\le t\le 1 }  |Y_t|   \le \theta  \big\} 
   \le  
    e   \, \P\big\{| g|< \frac{\theta}{s(J)}\big\} ^{\frac{1}{p}}   ,$$
 By taking logarithms in both sides, we get 
$$  p \ge \frac{ \log \displaystyle{ \frac{1}{\P \{|g|< \frac{\theta}{ s(J)} \} } } } {1+ K\theta^{-2}   } .  $$
Consequently,  
$$   \int_0^1    \big|\sum_{n\in  J}   \frac{1}{n^2} \cos   nt\big|   \dd t\ge  
\frac{s(J)^2\theta^{ 2} } {\theta^{ 2}+ K }\log
 \frac{1}{\P \{| g|<  {\theta}/{ s(J)} \}} .  $$
In particular, if $\theta =s(J)$,  
$$   \int_0^1    \big|\sum_{n\in  J}   \frac{1}{n^2} \cos   nt\big|   \dd t\ge  
C \frac{s(J)^4  } {s(J)^{ 2}+ K   } \ge C\big(\sum_{n\in  J}   \frac{1}{  n^2}\big)^{2} ,  $$
 since $s(J)^{ 2}\le \pi^2/6$. 
\end{proof}

\section{ Stationary sequences with Szeg\"o    spectral type conditions}
 \label{toep.szeg} Let $ (X_1, \ldots, X_n)$  be a Gaussian   vector with associated covariance    matrix (or     Gram matrix)   $\Gamma=
\{\g_{i,j}\}_{1\le i,j\le n}
$. 
  Assume that 
$\Gamma $ is invertible and let 
$\Gamma_j=\Gamma({X_1,\ldots, X_j })$ be the $j$-th principal minor of $\Gamma $. Define $\rho_j= \det(\Gamma_{j-1})/\det(\Gamma_{j })$,
$j=1,\ldots, n $,
$\Gamma_0=1$.   
  By  Gram-Schmidt orthogonalization process  we   obtain from $X_1, \ldots, X_n$ an orthogonal sequence $Y_1, \ldots, Y_n$, which may
be expressed as follows 
 \begin{eqnarray*}Y_j={1\over \sqrt{\Gamma_{j-1}\Gamma_j }}\left|\begin{matrix}
\g_{1,1}       &\g_{2,1}            &\cdots          &\g_{j,1}\cr
\g_{1,2}    &\g_{2,2}      &\cdots           &\g_{j,2}\cr
 \vdots         &   \vdots                  &\ddots      &\vdots \cr
\g_{1,j-1}        &\g_{2,j-1}                   &\cdots      &\g_{j,j-1}\cr
X_1        &X_2                  &\cdots        &X_j
\end{matrix} \right|\qq j=1,\ldots, n.
\end{eqnarray*}
Developing along the last line gives 
$$Y_j=  L(X_1,\ldots, X_{j-1}) + \rho_{j } X_j \qq j=1,\ldots, n   .$$
  From this  and (\ref{ks}),  we easily deduce
the  following  basic estimate: for
$z_j>0$ arbitrary,
\begin{equation}\label{basic} \prod_{j=1}^n \Big(  \int_{-z_j }^{z_j }
e^{-x^2/2}\frac{\dd x}{\sqrt{2\pi}} \Big)\le \P\Big\{ \sup_{j=1}^n \frac{|X_j|}{ z_j}\le 1 \Big\}\le \prod_{j=1}^n \Big( 
\int_{-z_j\sqrt{\rho_j}}^{z_j\sqrt{\rho_j}} e^{-x^2/2}\frac{\dd x}{\sqrt{2\pi}} \Big)
 .
\end{equation} 
  The search of suitable bounds of $\rho_j$ is consequently a fundamental question. There are some special inequalities involving the
Gram determinants $\det(\Gamma_{j })$. For instance (\cite{Ku}, p.382--383),
\begin{equation}\label{kurepa0} \det \Gamma({X_1,\ldots, X_j })\le  \prod_{i=1}^j \|X_i\|_2^2.
\end{equation} 
\begin{equation}\label{kurepa} \det \Gamma({X_1,\ldots, X_j })\le \det \Gamma({X_1,\ldots, X_k })\det \Gamma({X_{k+1},\ldots, X_j }) .
\end{equation} 
Hence, \begin{equation}\label{kurepa1}  \rho_j \ge \frac1{\|X_j\|_2^2} .\end{equation} 

See the upper bound (\ref{kurddp}), see also \cite{D},\cite{E}. If   
$\{X_j, j\in \Z\}$ is a Gaussian stationary   sequence with spectral  function $F $,     it is natural to wonder which
spectral conditions may be imposed on $F$ to get   upper and lower bounds to the probability $\P\big\{ \sup_{j=1}^n |X_j|\le z\big\}$ (or
to its logarithm), which are comparable and remain valid for some range of values of type $0<z\le z_0$, $n\ge n_0$. 
 Let 
\begin{equation}\label{coeff}c_n= \frac{1}{2\pi}\int_{-\pi}^\pi e^{-in\l} F(\dd \l),
\end{equation} 
so that  $\E X_jX_k= c_{j-k}$. The corresponding Hermitian forms are also called the  Toeplitz forms  associated   with $F$, and we have
the representation
$$ T_n=\sum_{\m, \nu=0}^n c_{\nu-\m}u_\m\bar u_\nu= \frac1{2\pi}\int_{-\pi}^\pi \big|u_0+u_1e^{i\l}+u_2e^{2i\l}+\ldots u_ne^{in\l}\big|^2
F(\dd
\l).$$ Recall that $F$ is said of finite type if its range
consists of a finite number of values. In the opposite case, it is called of infinite type.
 The  forms $T_n$ are positive definite unless $F$ is of finite type (\cite{GS}, \S1.11). If $F$ is of infinite type, all determinants of
the forms
$T_n$ are positive, namely $\det \G_n>0$ for all $n$.       
  \begin{theorem} \label{szego} 
Assume
$F$ is of infinite type. Let 
$f$ be the  Radon-Nycodim derivative  of the absolutely continuous part of $F$, and put   
\begin{equation*}  G(f)  =\begin{cases}\exp\Big\{ \frac1{2\pi}\int_{- \pi}^{ \pi}
\log f(t) \dd t\Big\} &\quad {\rm if}\  \log f(t) \ {\rm is\  integrable}\cr 
0 &\quad {\rm otherwise}.
\end{cases}
\end{equation*}
Then for   all
$n
$ and
$z>0$,
$$ \Big( 
\int_{-z }^{z } e^{-x^2/2}\frac{\dd x}{\sqrt{2\pi}} \Big)^n\le   \P\Big\{ \sup_{j=1}^n |X_j|\le z \Big\}\le   \Big( 
\int_{-z/\sqrt{G(f)}}^{z/\sqrt{G(f)}} e^{-x^2/2}\frac{\dd x}{\sqrt{2\pi}} \Big)^n
 .
$$
 \end{theorem}  
\begin{remark}\label{h2} More explicit  formulations  can be deduced from estimate (\ref{basicsmill}). The quantity $\exp\big\{
\frac1{2\pi}\int_{-
\pi}^{
\pi}
\log f(t) \dd t\big\}$ is by definition   the geometric mean of $f$. The condition that $\log f$ be   integrable is satisfied by a
remarkable class of functions. Let $u(z)=
\sum_{n=0}^\infty c_n z^n$ be an analytic function, regular in the open unit disk
$|z|<1$  and belonging to  
$  H_2$, namely the integral 
$$ \frac1{2\pi} \int_{-\pi}^\pi  |u(re^{it})|^2 \dd t$$
is bounded for every $r<1$. This is equivalent to the fact that $\sum_{n=0}^\infty |c_n|^2<\infty$.  Then the limit
$$\lim_{r\to 1-0}    u(re^{it})= h(t)$$
exists for almost every $t$. Let $f(t)= |h(t)|$. We   furthermore have that $\log f$ is  (Lebesgue) integrable  and (see  \cite{GS}, 
\S1.13),
$$\frac1{2\pi} \int_{-\pi}^\pi  |\log |f( e^{it})|| \dd t\le  \sum_{n=0}^\infty |c_n|^2.$$
 \end{remark}
 \begin{proof} According to (\ref{ks}), only the second  inequality has to be proven. We
have the explicit formula 
 \begin{equation} 
\frac1{\rho_j}=\frac{\det(\Gamma_{j })}{\det(\Gamma_{j -1})}= \min_{p}\int|p(e^{i\l} |^2 \m(\dd \l) ,
\end{equation}   
where the minimum is taken over all polynomials $p$ of degree $j -1$, of type $a_0+ a_1z +\ldots + a_jz^{j -1} $ with $|a_{j -1}|=1$. 
See (\cite{GS} \S 3.1.a  and \S 2.2.a). Further, when $j$ tends to infinity, these minima are decreasing and in fact
$$ \frac{\det(\Gamma_{j })}{\det(\Gamma_{j -1})}\  \downarrow \ \exp\Big\{ \frac1{2\pi}\int_{- \pi}^{ \pi}
\log f(t) \dd t\Big\} . $$
 Consequently, by (\ref{basic}) and monotonicity, 
\begin{equation}\label{basic1}   \P\Big\{ \sup_{j=1}^n |X_j|\le z \Big\}\le   \Big( 
\int_{-z/\sqrt{G(f)}}^{z/\sqrt{G(f)}} e^{-x^2/2}\frac{\dd x}{\sqrt{2\pi}} \Big)^n
 .
\end{equation} 
\end{proof}
  \begin{remark}  A direct use of (\ref{eig.weyl}) would have provided  a less precise result. Much later, Szeg\"o  also showed   that a
rate of convergence can be associated to (\ref{eig.weyl})  in presence of reasonable smoothness assumptions. Suppose that
$f$ has a derivative which satisfies a Lipschitz condition of order $\a$, $0<\a<1$. Then,
\begin{equation}\label{sweyl1}\lim_{n\to \infty}\Big[\log \det \Gamma_n -\frac{n+1}{2\pi}\int_{-
\pi}^{
\pi} \log f(t) 
\dd t \Big]= \frac1{\pi}\int\int |h(z)|^2\dd \s,
\end{equation}
where the function $h(z)$ is analytic in $z$ and is defined by the equality
$$ h(z)= \frac1{4\pi}\int_{-\pi}^\pi \log f(\l)\frac{1+ze^{-i\l}}{1-ze^{-i\l}} \dd \l,$$
and the integration in the right-handside in (\ref{sweyl1}) is along the unit circle. See   \cite{Lib} for some generalization.
\end{remark}
 \begin{example} Assume that the spectral density exists,
   $ f(t)  =a_0+   \sum_{n\in \Z*} a_n  e^{int}$, $a_{-n}=a_{|n|}$,  and
\begin{equation}\label{spcoecond} \sum_{n\in \Z*} |a_n| < |a_0| .
\end{equation}
Let $m$ and $M$ denote the essential 
lower and upper bound $f$ respectively.  Then $m>0$. The conclusion of Theorem \ref{szego} holds. The link between a  (square
integrable) spectral function and its corresponding correlation function being  given by
\begin{equation}\label{linkcorspf} f(t)  =\E X^2_0+    \sum_{n\in \Z*} (\E X_0X_{|n|}) e^{int},
\end{equation}
this holds in particular for  the Ornstein-Uhlenbeck sequence $\{U(n), n\ge 0\}$. Indeed, in this case $f(t)=1+   \sum_{n\in \Z*}
e^{-|n|/2}  e^{int}$.
\end{example}
\begin{remark}In the lacunary case $ f(t)  =a_0+   \sum_{k\in \Z*} a_k  e^{in_kt}$, $n_k= \l^k$, $\l>1$,   $a_{-k}=a_{|k|}$, condition $ 
m>0$    is   equivalent to  (\ref{spcoecond}), since the maxima of the polynomials
$\sum_{|k|\le N} a_k  e^{in_kt}$ verify  
$$ \sup_{-\pi\le t\le \pi}\Big|\sum_{|k|\le N} a_k  e^{in_kt}\Big|=  \sum_{|k|\le N} |a_k| .
 $$ 
  This follows from a well-known 
theorem of Sidon.
\end{remark}

\begin{example}
Let $b>0$, and consider again    $Y_j =  U(jb)-U((j-1)b) ,    j=1,2,\ldots  $. We   compute
  the corresponding  geometric mean.  Recall that 
 \begin{equation}\E Y_\ell Y_{\ell+u}=\begin{cases}2 (1-  e^{- b/2}) \quad  & {\rm if}\ u=0,\cr 
 (2  -e^{ b/2} -e^{- b/2})e^{-ub/2} \quad  & {\rm if}\ u\ge 1.
\end{cases}\end{equation} 
   Let $r=e^{-b/2} $. Then $ \d:= 2  -e^{ b/2}-e^{- b/2}=-\frac{(1-r)^2}{r} \sim    - {b^2 \over 4}  + \mathcal
O(b^3) $, as $b\to 0$. Now introduce the Poisson kernel
\begin{equation}\label{poisson} g(t)= \sum_{n\in \Z} r^{|n|} e^{int}= \frac{1-r^2}{1-2r\cos x +r^2}, \qq 0<r<1.
\end{equation}
  It is well-known that $\log g(t)$  is integrable. Further, 
\begin{equation}\label{poisson1}\int_0^{\pi} \log (1-2r\cos x +r^2)\dd x =0,\quad \qq (=\pi \log r^2\ {\rm if}\ r>1).
\end{equation} 
 The  spectral function, call it $h(t)$, verifies  
\begin{eqnarray}\label{poisson2}h(t)&=&   2(1-r)+ 
 \d \sum_{u\in\Z^*}r^{|u|} e^{iut} 
= 2(1-r)-\d + \d
\sum_{u\in \Z }r^{|n|} e^{iut}
  \cr &=&2(1-r)+  \frac{(1-r)^2}{  r} -\frac{(1-r)^2}{  r}  \frac{ (1-r^2)}{1-2r\cos x +r^2}
 \cr &=&   \frac{ 1-r ^2}{r}\Big(1 -   \frac{ (1-r)^2}{1-2r\cos x +r^2}\Big)
 \cr &=&\frac{2(1-r^2)(1-\cos x)}{1-2r\cos x +r^2} . 
\end{eqnarray} 
 
We have from (\ref{poisson1})
\begin{eqnarray*}\int_{-\pi}^{\pi} \log h(t) \dd t&=& \int_{-\pi}^{\pi} \log \big(2(1-r^2)(1-\cos x)\big) \dd t
\cr &=& 2\pi \log [2(1-r^2)]+
\int_{-\pi}^{\pi} \log
 (2\sin^2 \frac{t}{2}) \dd t
\cr &=& 2\pi \log [4(1-r^2)]+
4\int_{0}^{\pi} \log
  \sin  \frac{t}{2}  \dd t. 
\end{eqnarray*}
But $\int_{0 }^{\pi} \log
  \sin  \frac{t}{2}  \dd t =-\pi \log 2$. Therefore
$$\frac1{2\pi}\int_{-\pi}^{\pi} \log h(t) \dd t=\log [2(1-r^2)]+\log 2- 2\log 2=\log  (1-r^2).
$$
 
Thus $ G(h)=1-r^2= 1-e^{-b}$ and by Theorem \ref{szego},
\begin{equation}\label{uhlenb}  \P\Big\{ \sup_{j=1}^n \big| U(jb)-U((j-1)b)\big|  \le z \Big\}\le   \Big( 
\int_{-\frac{z}{\sqrt{1-e^{-b}}}}^{\frac{z}{\sqrt{1-e^{-b}}}} e^{-x^2/2}\frac{\dd x}{\sqrt{2\pi}} \Big)^n
 .
\end{equation} 
\end{example} 
\vskip 2pt 
 \vskip 5pt 
 More generally, let $\xi(t), t\ge 0$ be a Gaussian  stationary process  and $b$ being a positive real, let
 $\xi_b(j)= \xi((j+1)b)-\xi(jb), \ j=0,1, \ldots  $
Let also     
$\g(h)= \E \xi(0)\xi(h)$, $\g(0)=1$, $\s(h)= \sqrt{2(1-\g(h)}$. \begin{proposition}
Assume that $\g(h)$ is convex decreasing and let $f= -\g'$. Then $f(t)=  2\sin(\frac{t}{2}) \,g(t)$ where $$g(t)\asymp  
\frac{\s^2(b)+\s^2(2b)+\ldots +
\s^2((m-1)b)}{m} +  \s^2(mb)    $$ as $t\to +0$, and we write $m=\lfloor \frac{\pi}{|t|} \rfloor$ for brevity. Further $\log f$ is
integrable.  
\end{proposition}  
\begin{proof}  
   Set $\D_n=\s^2(nb)-\s^2((n-1)b),\  n=1,2,\ldots$,   $ \D_0=0$. Then
$$\E \xi_b(0)\xi_b(n)=\frac1{2} \big\{-2\s^2(nb)+\s^2((n-1)b)+\s^2((n+1)b) \big\}=\frac1{2} \big\{\D_{n+1}-\D_n 
\big\},$$
and 
 $f(t)  =\E \xi_b(0)+    \sum_{n\in \Z*} (\E \xi_b(0)\xi_b(|n|)) e^{int}= \sum_{n=0}^\infty (\D_{n+1}-\D_n)\cos nt$.  
 By using Abel summation and the relation $\cos jt-\cos(j+1)t=2\sin\frac{t}{2}\, \sin\frac{(2j+1)t}{2} $, $f(t)$ can be rewritten
as  $f(t)= 2\sin\frac{t}{2}\, g(t)$ where  
$$g(t)= \D_1\sin\frac{3t}{2}+\D_2\sin\frac{5t}{2}+ \ldots $$
As $\s^2$ is concave increasing, it follows that   $\D_1\ge \D_2\ge \ldots$ The behavior of sine
series  with   non increasing coefficients were studied by Salem. We refer to       Popov's article \cite{Po} for instance, for
the result below (Telyakovskii's estimate) and recent sharpenings,
$$g(t)\asymp t \sum_{k=1}^{\lfloor \frac{\pi}{|t|} \rfloor}k\D_k  
\qq \quad t\to +0  .$$
The constants involved in the symbol $\asymp$ are absolute. By using again Abel summation, 
$$ \sum_{k=1}^{m}k\D_k = \s^2(b)+\s^2(2b)+\ldots + \s^2((m-1)b) + m\s^2(mb).$$
Letting $m=\lfloor \frac{\pi}{|t|} \rfloor$, we deduce 
 $$g(t)\asymp   \frac{\s^2(b)+\s^2(2b)+\ldots + \s^2((m-1)b)}{m} +  \s^2(mb)  ,$$
as $t\to +0$. Therefore 
$$f(t)\asymp  \sin\frac{t}{2}\qq \quad t\to +0  . $$
It follows that $\log f$ is integrable, as
claimed. 
\end{proof} 
\begin{remark} Assume   $f$ be integrable, $f\not\equiv 0$. The condition that $\log f$ be integrable characterizes  the fact that
there exists an
 analytic function $h(z)$ of the class $H_2$ such that $f(t)= |h(z)|^2$, $z=e^{it}$.
This is well-known extension of Fej\'er-Riesz's representation theorem     for non negative trigonometric polynomials.   
   It also characterizes the property that $\{ X_j, j\in \Z\}$ be non-deterministic.
\end{remark}

We conclude this section with an abstract and less handable reformulation of Theorem \ref{szego}. Recall that
$\Gamma_j=\Gamma({X_1,\ldots, X_j })$. Let
$E_k$ be the  subspace of
$L^2$ linearly generated by
$X_1,
\ldots, X_{j-1}$ and put
  $$ \t_j= \| X_j- E_{j-1}  \|  ,$$
namely the distance from $X_j$ to $ E_{j-1} $. 
\begin{proposition} i) Let $(X_1, \ldots, X_n)$ be Gaussian with invertible covariance matrix. Then
$$\P\Big\{ \sup_{j=1}^n \frac{|X_j|}{ z_j}\le 1 \Big\}\le \prod_{j=1}^n \Big( 
\int_{-\frac{z_j}{\t_j}}^{ \frac{z_j}{\t_j}} e^{-x^2/2}\frac{\dd x}{\sqrt{2\pi}} \Big)
 .$$
ii) Let $\{X_j, j\in \Z\}$ be a Gaussian stationary
 seqence having an absolutely continuous spectrum, with spectral density function $f$.  Then  $\t_j^{-2}=\sum_{k=0}^j|\p_k(0)|^2$ where $\{\p_k, k\in \Z\}$
is  the orthonormal sequence of   polynomials associated to the weight function $f(x) $.  
\end{proposition}
\begin{proof} First notice that 
\begin{equation}  \Gamma_j= \Gamma_{j-1} \t_j^2.\end{equation} 
 For a reference, see \cite{AG} p.13. According to (\ref{basic}),
$$\P\Big\{ \sup_{j=1}^n \frac{|X_j|}{ z_j}\le 1 \Big\}\le \prod_{j=1}^n \Big( 
\int_{-\frac{z_j}{\t_j}}^{ \frac{z_j}{\t_j}} e^{-x^2/2}\frac{\dd x}{\sqrt{2\pi}} \Big)
 .
$$ 
As to b), this follows from \cite{GS}, (10) p.40.\end{proof}
      
See
also  \cite{Lif0}, Proposition 3, Section 3 where a more complicated proof   is given.   For applications of strong Szeg\"o limit
theorems to linear prediction of stationary processes, we refer   to  Chapter 10 of
\cite{GS}, which is entirely devoted to this question.  
    

\section{Ger\v sgorin's Disks and Matrices with Dominant
Principal Diagonal} \label{section[ddp]}
  In this part, we are rather concerned with the non-stationary case.   For an important class of matrices   the parameters  $\rho_j$ in
(\ref{basic}) turn  up to be easily controlable.  An
$n\times n$ matrix
$A=\{ a_{i,j}, 1\le i,j\le n\}$   has dominant principal diagonal if 
\begin{equation}\label{ddp}|a_{i,i}|>\sum_{j=1\atop 
j\not =i}^n |a_{i,j }|, \qq \quad i=1,2,\ldots, n.
\end{equation} This notion already appeared in Minkowski  and Hadamard works (see the overview in \cite{Ta}). Matrices with dominant
principal diagonal define a quite remarkable class: they are invertible and their determinants are   easy to estimate. 
Put for  $i=1,2,\ldots, n$,
 $$ A_i=\sum_{j=1\atop 
j\not =i}^n |a_{i,j }|,\qq m_i=|a_{i,i}|- A_i,\qq M_i=|a_{i,i}|+ A_i. $$
The following basic estimate is due to Price (\cite{P}, Theorem 1), the lower bound being previously proved by Ostrowski in \cite{O},
(see also    \cite{B},\cite{FV},\cite{H},\cite{Ky} for various refinements).
\begin{equation} \label{price}0<m_1 \ldots m_n\le |\det(A)|\le M_1 \ldots M_n .
\end{equation} 
 
If $A$ is a Gram matrix, it follows from this and inequality (\ref{kurepa1}) that 
   \begin{equation}\label{kurddp}   \rho_j 
 \le  \frac{  (1+\tau)^{j-1}}{  a_{j,j}}, \qq\quad {\rm where}  \quad \tau= \max_i \frac{A_i}{a_{i,i}}<1.
   \end{equation} 
Then by (\ref{kurddp}) and (\ref{basic}), 
\begin{equation}\label{kurddp1}    \P\Big\{ \sup_{j=1}^n \frac{|X_j|}{ z_j}\le 1 \Big\}\le \prod_{j=1}^n \Big( 
\int_{ -  \frac{z_j   }{ \sqrt{ a_{j,j}}}(1+\tau)^{\frac{j-1}{2}}} ^{   \frac{z_j   }{ \sqrt{ a_{j,j}}}(1+\tau)^{\frac{j-1}{2}}} 
e^{-x^2/2}\frac{\dd x}{\sqrt{2\pi}} \Big)
 .
 \end{equation} 
This can be however improved. 
\begin{proposition} \label{basicddpp}   Let $(X_1,\ldots, X_n)$ be a Gaussian vector  and assume that for some $r<1$,
\begin{equation}\label{basicddp}     \sum_{j=1\atop j\not = i}^n|\E X_iX_j|\le r \, \E X_i^2, \qq\quad i=1,\ldots, n.  \end{equation} 
Then, 
$$   \P\Big\{ \sup_{j=1}^n  |X_j|     \le z \Big\}\le  \prod_{j=1}^n\P\Big\{  |X_j|     \le  \frac{z   }{\sqrt {1-r}} \Big\}  
 .
$$ 
\end{proposition}
Our result much improves Theorem 2.2  in \cite{LS2} where only a bound of $\sup_{j=1}^n   X_j $ is given under similar assumptions
(assumption (2.4) has to be modified). 
The proof uses the following general   estimate  for quadratic forms
\begin{lemma} \label{quadraf}
$$ \sum_{ i=1}^n x_i^2\Big(a_{i,i}+\sum_{j=1\atop 
j\not =i}^n |a_{i,j }|\Big)\ge \sum_{ i,j=1}^n x_ix_ja_{i,j}\ge \sum_{ i=1}^n x_i^2\Big(a_{i,i}-\sum_{j=1\atop 
j\not =i}^n |a_{i,j }|\Big).
$$\end{lemma}
\begin{proof} At first  we have
$$\Big| \sum_{1\le i<j\le n}  x_ix_ja_{i,j}\Big|\le   \sum_{1\le i<j\le n}  \Big(\frac{x_i^2+x_j^2}{2} \Big)|a_{i,j} |\le \frac1{2} 
\sum_{j=1}^n x_j^2 \Big(\sum_{\ell=1\atop
\ell\not= j}^n|a_{j,\ell} |\Big)  .$$
And next  
\begin{eqnarray*}\sum_{ i,j=1}^n x_ix_ja_{i,j}=\sum_{ i=1}^n   x_i^2a_{i,i}+2\sum_{1\le i<j\le n} x_ix_ja_{i,j}  
 \ge  \sum_{ i=1}^n   x_i^2\Big(a_{i,i}- 
\ \sum_{\ell=1\atop
\ell\not= i} ^n|a_{i,\ell} |\Big).
\end{eqnarray*}
This yields the right-inequality. The left-one  follows similarly.
\end{proof} 
\begin{proof}[Proof of Proposition \ref{basicddpp}]   
    Let   $ 0< 
\a<1  -r$,
$\uY= 
\{\sqrt  \a\, a_{i,i}\, g_i,1\le i
 \le n\}$   and   $B= \{ \E X_iX_j- \E Y_iY_j, 1\le i,j\le n\}$. By applying Lemma \ref{quadraf}  to $B$, we get 
\begin{eqnarray*}\sum_{i,j=1}^n x_ix_j (\E X_iX_j- \E Y_iY_j)& =&\sum_{i,j=1}^n x_ix_jb_{i,j}
= \sum_{ =1}^n x_i^2  \Big(a_{i,i}
(1-\a)-\sum_{j=1\atop  j\not =i}^n |a_{i,j }|\Big)\cr &\ge& \sum_{ i=1}^n x_i^2  a_{i,i} (1-\a -r) \ge 0.
\end{eqnarray*}
Thus by   Anderson's inequality \cite{To} p.55, for any convex set $C$ symmetric around $0$,
\begin{equation}\label{ddpb}  \P\big\{\uX\in C\big\}\le \P\big\{\uY\in C\big\} .
\end{equation}
By choosing $C=\big\{ (x_1,\ldots, x_n)\in \R^n:|x_i|\le z\big\} $, we deduce 
 $$   \P\Big\{ \sup_{j=1}^n  |X_j|     \le z \Big\}\le \P\Big\{ \sup_{j=1}^n  |Y_j|     \le z \Big\}=  \prod_{j=1}^n\P\big\{  |  
a_{i,i}  g_i|    
\le  \frac{z   }{\sqrt  \a} \big\}  .
$$ 
Letting next 
$\a$ tend to
$1-r$, finally leads to
$$   \P\Big\{ \sup_{j=1}^n  |X_j|     \le z \Big\}\le   \prod_{j=1}^n\P\big\{  |  
a_{i,i}  g_i|    
\le  \frac{z   }{\sqrt   {1-r}} \big\}  
 =  \prod_{j=1}^n\P\big\{  |X_j|
\le  \frac{z   }{\sqrt  {1-r}} \big\}  
 .
$$ 
as claimed. \end{proof}
 \begin{remark} Condition (\ref{ddp})
 has to be related with famous Ger\v sgorin's theorem, which states that the eigenvalues of an
$n\times n$ matrix with complex entries lie in the union of the closed disks (Ger\v sgorin disks)
\begin{equation}|z-a_{i,i}| \le A_i \qq\quad  (i=1,2,\ldots , n) \label{Ger} 
  \end{equation} 
in the complex plane. This result has naturally many concrete applications. An example   to the analysis of flutter phenomenon in aircraft
design  is described in \cite{Ta}. There is an analog result due to Brauer    on ovals of Cassini stating that 
\begin{equation}\label{Ger1}   {|z-a _{i,i}||z-a_{j,j}| }\le A_iA_j \qq\quad  (i,j=1,2,\ldots , n,\ i\not = j).  
  \end{equation} 
See   \cite{Br}. In  relation with this, we have that if 
\begin{equation}\label{ddp2}|a_{i,i}||a_{k,k}|>A_i A_k   \qq\quad  (i,j=1,2,\ldots , n,\ i\not = j) 
\end{equation}
then $\det(A)>0$. Note that the relations (\ref{ddp2}) imply $|a_{i,i}| >A_i $ for all $i$ but one. 
 \end{remark}      
       
Matrices with dominant
principal diagonal are used in a crucial way in
\cite{M}  starting from (\ref{basic}), see proof of Lemma
2. Assume  
$\g(u)$   is convex on $[0,\d]$ for some
$\d>0$, and let $\s^2(x)= 2(1-\g(x))$. Let also 
$t_0<t_1<\ldots<t_n$ with $t_n-t_0\le \d$. By using convexity of $\gamma$, $  \E (X(t_i)-X(t_{i-1})(X(t_j)-X(t_{j-1})\le 0$,
so that 
\begin{eqnarray*} & &A_i:=\sum_{j=1\atop j\not= i }^n  |\E (X(t_i)-X(t_{i-1})(X(t_j)-X(t_{j-1}) |
\cr&=&-   \E (X(t_i)-X(t_{i-1})\big[X(t_n)-X(t_{i}) +   X(t_{i-1})-X(t_0)\big] 
\cr&=& \s^2(t_{i}-t_{i-1})+\frac1{2}\big[\s^2(t_{i-1}-t_{0})-\s^2(t_{i }-t_{0})+\s^2(t_{n}-t_{i})-\s^2(t_{n}-t_{i-1})
\big]
\cr&<&\s^2(t_{i}-t_{i-1}).\end{eqnarray*} However the ratio $A_i/\s^2(t_{i}-t_{i-1})$ has to be estimated in order to adjust
with assumption   (\ref{basicddp}), and we don't see how this can be done. It seems therefore that inequality (7) (and thereby  (8))   
in
\cite{M} needs a correction.  A strictly weaker estimate can be deduced from
(\ref{kurddp1}). A comparable estimate (without absolute values) however trivially follows from Slepian's lemma since the process has
negatively correlated increments, see \cite{LS} Theorem 4.5. 
\vskip 5pt
\noi{\it Final remark.}  Although not presented here, the results from the sections 2,3,4,5 admit some extensions to Gaussian random
fields defined on $\R^n$ with values in $\R^d$. 
   


\begin{thebibliography}{99}
\bibitem{AG}  Achieser N.I., Glasman I.M.: (1954)  {\sl  Theorie der Linearen Operatoren in Hilbert-Raum}, Berlin.
 \bibitem{AL}  Aurzada F.,   Lifshits M.: (2008)  {\sl Small deviation probability via chaining}, Stoch. Proc. \&   Appl. {\bf 118},   2344--2368.
  1--26.    
   \bibitem{Ba}{Barthe F.}: (1998)    {\sl On a reverse form of the Brascamp--Lieb inequality},  Invent. math.   {\bf 134},
 335--361.
 \bibitem{Bra}{Brascamp H.M., Lieb E.H.}: (1976)    {\sl Best Constants in Young's Inequality, Its Converse, 
and Its Generalization to More than Three Functions},  Advances in Math.   {\bf 205},
 151--173. 
 \bibitem{Br}{Brauer F.L.}: (1947)    {\sl  Limits for the characteristic roots of a matrix II},  Duke Math. J.   {\bf 14},
 21--26. 
 \bibitem{B}{Brenner J.L.}: (1954)    {\sl A Bound for a Determinant with Dominant Main Diagonal},  Proc. Amer. Math. Soc.   {\bf 5} No4,
 631--634. 
  \bibitem{BC} {Be\'ska M.,  Ciesielski Z.}:   (2006)  {\sl Gebelein's inequality and its consequences}, Approximation and Probability, Banach Center Publications   {\bf 72}, Institute of Mathematics, Polish Academy of Sciences, Warszawa 11--23.

\bibitem{C}  {Cs\'aki E.}: (1994)    {\sl Some limit theorems for empirical processes}, in: Vilaplana, J.P., Puri, M.I. (Eds), Recent
Advances in Statistics and Probability (Proc. {\bf  4th} IMSIBAC) VSP, Utrecht, 247--254. 
 \bibitem{D}{Davies P.J.}: (1965)    {\sl Interpolation and Approximation}, New-York-Toronto-London, 393p. 
\bibitem{E} Everitt W.N.: (1962)   {\sl Inequalities for Gram determinants},
Quart. J. Math. Oxford Ser. ({\bf  2}) {\bf 8}, 191--196.
\bibitem{FV} Feingold D.G., 
Varga R.S.: (1962)   {\sl Block diagonally dominant matrices and generalisations of the Ger\v sgorin circle theorem}, Pacific J. Math. 
{\bf  12}, 1241-1250.
   
\bibitem{GS} Grenander U., Szeg\"o G.: (1958)   {\sl Toeplitz forms and their applications}, Univ. of California Press, Berkeley and Los
Angeles.
\bibitem{GRS} Guerra F., Rosen L., Simon B.:  (1975). {\sl The $P(\phi)_2$ Euclidean quantum field theory 
as 
classical statistical 
mechanics},  
Ann. of 
Math. {\bf 101}, 111--189.
\bibitem{H} Haynsworth E.V.: (1953)   {\sl Bounds for determinants with dominant principal diagonal}, Duke Math. J.  {\bf 20}, 199--209.
  \bibitem{KLS}{ Klein A., Landau L.J.,  Shucker D.S.}: (1982)    {\sl Decoupling
inequalities for stationary Gaussian processes}, Ann. of Prob.  {\bf 10},
 702--708. 
  \bibitem{Ko}{Konyagin S.V.}: (1981)    {\sl On the problem of Littlewood}, Izv. Acad Nauk SSSR Ser. Math. [Math.USSR-Izv.], {\bf 45} No.
2, 243--265.
  \bibitem{Ky}{Ky Fan}: (1971)    {\sl On the singular values of compact operators}, J. London Math. Soc.   {\bf 2} No3,
 187--189. 
 \bibitem{Ku}{Kurepa S.}: (1967)    {\sl Kona\v cno dimensionalni vectorski prostori i primjene}, Zagreb. 
  \bibitem{La} Lancaster H.O.: (1969)  {\sl The chi-squared distribution},  Wiley Publ. in Statistics.
  \bibitem{Lib}{ Libkind L.M.}: (1972)    {\sl  Asymptotic of the eigenvalues of Toeplitz forms}, Mat. Zametki {\bf 11} No. 2, 151--158.
\bibitem{Lif0}{ Lifshits M.}: (1995)   {\sl Gaussian random functions}, Dordrecht, Kluwer
Academic  Publishers. 
\bibitem{Lif}{ Lifshits M.}: (1999)    {\sl Asymptotic behavior of small balls probabilities,} in: Prob. theory and Math.  Statist. Proc. VII,
 International Vilnius Conference, VSP/TEV, 453--468.
   \bibitem{LS2}{Li Wenbo V., Shao Qi-Man}: (2004)    {\sl Lower tail probabilities for Gaussian processes},
 Ann. of Prob. {\bf 32} No1A, 216--242.
\bibitem{LS}{Li Wenbo V., Shao Qi-Man}: (2005)    {\sl Gaussian Processes: Inequalities, Small Balls Probabilities and Applications}, in  D.N. Shanbhag et
al. (Eds.), Stochastic Processes: Theory and Methods, in: Handb. Statist. {\bf 19}, Elsevier, Amsterdam 2001, 533--597.
 \bibitem{MPS}{McGehee O.C., Pigno L., Smith B.}: (1981)    {\sl Hardy's inequality and the $L_1$-norm of exponential sums}, Ann. of
Math. (2), {\bf 113} No. 3, 613--618.     \bibitem{M} Marcus M.B.: (1968)  {\sl Gaussian processes with stationary increments possessing discontinuous sample paths},
 Pacific J. of Math.  {\bf 26} No1,  149--157.  
     \bibitem{N} Newell G.F.: (1962)  {\sl Asymptotic extreme value distributions for one dimensional diffusion processes},
 J. Math. Mech.   {\bf 11} No1,  481--496.
  \bibitem{Ol}  Olevskii A.M.: (1975)   {\sl Fourier series with respect to
general orthogonal systems},  Ergebnisse der Mathematik und ihrer
Grenzgebiete, Band {\bf 86}. 
 \bibitem{O} Ostrowski A.M.: (1951)  {\sl Note on bounds for determinants with dominant principal diagonal},
 Proc. Amer. Math. Soc.   {\bf 3},  26--30.
  \bibitem{Po} Popov A.Y.: (2003)  {\sl Estimates of the sums of sine series with monotone coefficients of certain classes},
 Math. Notes  {\bf 74} No6,   829--840.
   \bibitem{P} Price B.: (1951)  {\sl Bounds for determinants with dominant principal diagonal},
 Proc. Amer. Math. Soc.   {\bf 2},   497--502.
    \bibitem{S} Stolz W: (1996)  {\sl Some small balls probabilities for Gaussian processes under non-uniform norms},
 J. Theor. Prob.  {\bf 9},   613-630.   
     \bibitem{[T]} Talagrand M.: (1993)  {\sl New Gaussian estimates for enlarged balls},
 Geom. and Funct. Anal.  {\bf 3},   502--526.   
   \bibitem{Ta} Taussky O.: (1988)  {\sl How I became a Torchbearer for Matrix Theory},
  Amer. Math. Monthly {\bf 95} No9 (Nov. 1988),   801--812.
   \bibitem{To} Tong. Y.L.: (1980)  {\sl Probability Inequalities in Multivariate Distributions},
  Academic Press, New-York.
 \bibitem{V} Veraar M.: (2009)  {\sl Correlation Inequalities and Applications to Vector-Valued Gaussian Random Variables and Fractional Brownian Motion},
  Potential Analysis,  {\bf 30}, 341--370.
 \bibitem{Vo} Voss J.: (2009)  {\sl Upper and Lower Bounds in Exponential Tauberian Theorems},
  arXiv:0908.0642v2.
  \bibitem{W2}  {Weber  M.}: (1989)    {\it  The supremum of Gaussian processes with a constant variance},  Probab. Theory Related Fields
{\bf 81}   no.4, 585--591.  
  \bibitem{W3}  {Weber  M.}: (2006)   {\it On a  stronger form  of Salem-Zygmund's inequality    
  for random trigonometric sums with examples}, Periodica Math. Hungar.  {\bf 52} (2), 73--104.
   \bibitem{W}  {Weber  M.}:    {\it Dynamical Systems and Processes}, European Mathematical Society Publishing
House, IRMA Lectures in Mathematics and Theoretical Physics {\bf 14}  
   xiii+761p, 2009.  
   \bibitem{W1}  {Weber  M.}:  (2010)  {\it On small deviations of   Gaussian processes using majorizing measures}, preprint available
at arXiv:1012.3614v1.  
  \bibitem{Z}  {Zygmund A.}:  (2002)  {\it Trigonometrical Series},  Third Edition, Cambridge Mathematical Library, Cambridge University
Press. 
     \end{thebibliography}
\end{document}